\documentclass[11pt]{article}
\usepackage[a4paper]{anysize}\marginsize{3.5cm}{3.5cm}{1.3cm}{2cm}
\pdfpagewidth=\paperwidth \pdfpageheight=\paperheight
\usepackage{amsfonts,amssymb,amsthm,amsmath,eucal}
\usepackage{pgf}
\usepackage{bbm}
\usepackage[shortalphabetic]{amsrefs}

\usepackage{tikz} %Used for drawing picture in LaTeX
\usepackage{float}
\usetikzlibrary{arrows}
\usepackage{subfigure}
\usepackage{caption}

\theoremstyle{plain}
\newtheorem{thm}{Theorem}[section]
\newtheorem{theorem}[thm]{Theorem}
\newtheorem*{theorem*}{Theorem}

\newtheorem{proposition}[thm]{Proposition}
\newtheorem{corollary}[thm]{Corollary}

\theoremstyle{definition}

\newtheorem{example}[thm]{Example}

\newtheorem{thevarthm}[thm]{\varthmname}

\newenvironment{varthm*}[1]{\trivlist\item[]{\bf #1.}\it}{\endtrivlist}

\renewcommand\geq{\geqslant}

\renewcommand\leq{\leqslant}

\newcommand\be{\begin{eqnarray*}}
\newcommand\ee{\end{eqnarray*}}

\newcommand\R{\mathbb R}

\newcommand\N{\mathbb N}

\newcommand\E{\mathbb E}
\renewcommand\P{\mathbb P}

\newcommand\calo{{\mathcal O}}

\newcommand{\eps}{\varepsilon}
\newcommand\lra\longrightarrow
\newcommand\newop[2]{\def#1{\mathop{\rm #2}\nolimits}}
\newop\log{log}
\newop\ord{ord}
\newop\gon{gon}
\newop\Exc{Exc}
\newop\Pic{Pic}
\newop\Area{Area}
\newop\Gal{Gal}
\newop\SL{SL}
\newop\Bl{Bl}
\newop\mult{mult}
\newop\mass{mass}
\newop\div{div}
\newop\codim{codim}
\newop\sing{sing}
\newop\Zeroes{Zeroes}
\newcommand\eqnref[1]{(\ref{#1})}

\newcommand{\equ}{\ensuremath{\,=\,}}
\newcommand{\dgeq}{\ensuremath{\,\geq\,}}
\newcommand{\dleq}{\ensuremath{\,\leq\,}}
\newcommand{\st}[1]{\ensuremath{\left\{ #1 \right\}   }}

\def\keywordname{{\bfseries Keywords}}%
\def\keywords#1{\par\addvspace\medskipamount{\rightskip=0pt plus1cm
\def\and{\ifhmode\unskip\nobreak\fi\ $\cdot$
}\noindent\keywordname\enspace\ignorespaces#1\par}}
\def\subclassname{{\bfseries Mathematics Subject Classification
(2000)}\enspace}
\def\subclass#1{\par\addvspace\medskipamount{\rightskip=0pt plus1cm
\def\and{\ifhmode\unskip\nobreak\fi\ $\cdot$
}\noindent\subclassname\ignorespaces#1\par}}

\newcommand\beginproof[1]{\trivlist\item[\hskip\labelsep{\em #1.}]}
\newcommand\proofof[1]{\beginproof{Proof of #1}}

\def\endproof{\hspace*{\fill}\endproofsymbol\endtrivlist}

\def\endproofsymbol{\frame{\rule[0pt]{0pt}{6pt}\rule[0pt]{6pt}{0pt}}}

\begin{document}

\author{\L .~Farnik, T.~Szemberg, J.~Szpond, H.~Tutaj-Gasi\'nska}
\title{Restrictions on Seshadri constants on surfaces}
\date{\today}
\maketitle
\thispagestyle{empty}

\begin{abstract}
   Starting with the pioneering work of Ein and Lazarsfeld \cite{EinLaz93}
   restrictions on values of Seshadri constants on algebraic surfaces have
   been studied by many authors \cites{Bau99, BauSze11, HarRoe08, KSS09, Nak05, Ste98, Sze12, Xu95}.
   In the present note we show how approximation involving continued fractions
   combined with recent results of K\"uronya and Lozovanu on Okounkov bodies
   of line bundles on surfaces \cites{KurLoz14, KurLoz15} lead to effective statements
   considerably restricting possible values of Seshadri constants. These
   results in turn provide strong additional evidence to a conjecture
   governing the Seshadri constants on algebraic surfaces with Picard number $1$.
\subclass{MSC 14C20 \and MSC 14J26 \and MSC 14N20 \and MSC 13A15 \and MSC 13F20}
\end{abstract}

%*****************************************************************************

\section{Introduction}
   Let $X$ be a smooth algebraic variety and let $L$ be a line bundle on $X$.
   For any point $x\in X$ the real number
   $$\eps(L;x)=\inf_{C\ni x}\frac{(L\cdot C)}{\mult_xC},$$
   where the infimum is taken over all irreducible curves $C$ passing through $x$,
   measures in effect the \textit{local positivity of $L$ at $x$}.
   We say that a curve $C\subset X$ is a \emph{Seshadri curve of $L$ at $x$}
   if $\eps(L;x)=(L\cdot C)/\mult_xC$.

   These numbers, the Seshadri constants, were introduces by Demailly in \cite{Dem92} in connection
   with his works on the Fujita Conjecture and they have become a subject
   of considerable interest ever since. The well known Seshadri criterion
   of ampleness gives a fundamental positivity restriction on the Seshadri
   constants of ample line bundles.
\begin{varthm*}{Seshadri criterion of ampleness}
   Let $X$ be a smooth algebraic variety and let $L$ be a line bundle on $X$.
   Then
   $$L \mbox{ is ample iff }\;\inf_{x\in X}\eps(L;x)>0.$$
\end{varthm*}
   It is natural to wonder if there are any other constrains on the values
   of Seshadri constants of ample line bundles. Whereas examples of Miranda
   and Viehweg show that the Seshadri constants of ample line bundles can become
   arbitrarily small, in the groundbreaking paper \cite{EinLaz93} Ein and
   Lazarsfeld showed that there is a positive uniform lower bound when
   restricting to very general points. Oguiso in \cite{Ogu02} showed that the
   Seshadri function
   \begin{equation}\label{eq:Seshadri function}
      \eps(L;\;\cdot\;): X\ni x\lra \eps(L;x)\in\R
   \end{equation}
   is lower semi-continuous in the topology whose closed sets are countable
   unions of Zariski closed sets. In particular there is an open and dense
   in this topology subset of $X$ where the Seshadri function attains its
   maximal value. We denote this maximal value by $\eps(L;1)$. The number $1$
   here indicates that the Seshadri constant is taken at a very general point
   of $X$ without specifying this point. In this terminology the aforementioned
   result of Ein and Lazarsfeld is the following.
\begin{theorem*}[Ein-Lazarsfeld]
   Let $X$ be an algebraic surface and let $L$ be an ample line bundle on $X$.
   Then
   $$\eps(L;1)\geq 1.$$
\end{theorem*}
   This result cannot be improved in general even under the assumption that
   the self-intersection $d=(L^2)$ of $L$ is very large, see Example \ref{ex:L^2 large eps 1}.
   The main result of this note shows that nevertheless the set of potential
   values $\eps(L;1)$ can take on is surprisingly limited.

   For a non-square integer $d$ and $(p,q)$ a solution of the Pell equation
   \begin{equation}\label{eq:Pell}
      y^2-dx^2=1
   \end{equation}
   we define the following set
   $$\Exc(d;p,q)=\left\{1,2,\ldots,\lfloor\sqrt{d}\rfloor\right\}\cup\left\{\frac{a}{b}\;
      \mbox{ such that }\; 1\leq \frac{a}{b}<\frac{p}{q}d\;\mbox{ and }\; 2\leq b<q^2\right\}.$$
\begin{varthm*}{Main Theorem}
   Let $X$ be a smooth projective surface, $x\in X$, $L$ an ample line bundle on $X$ such that $(L^2)=d$ is not a square.
   Let $(p,q)$ be an arbitrary solution of the Pell equation \eqnref{eq:Pell}.
   Then either
$$
   \eps(L;1) \geq \frac{p}{q} d,
$$
   or $\eps(L;1)\in\Exc(d;p,q)$.
\end{varthm*}
   The finiteness of possible values of $\eps(L;1)$ strictly below any rational number smaller than
   $\sqrt{(L^2)}$ follows already from
   \cite{Ogu02}*{Theorem 1}. However Oguiso result addresses all line bundles $L$ \emph{separately}
   and the statement is ineffective. The key point of the Main Theorem is that possible values
   of $\eps(L;1)$ depend in a \emph{uniform} way only on $(L^2)$ and their set is effectively
   described and relatively small. Under additional assumptions on $X$ or $L$ one can thus
   try to reduce further the set of exceptional values (i.e. those lower than $p/q$).
   A typical assumption of this kind is that the
   Picard number $\rho(X)$ of $X$ equals $1$. The Seshadri constants of ample
   line bundles on surfaces with $\rho(X)=1$ were considered in a series
   of papers \cites{Ste98, Sze08, Sze12} leading to the following conjecture
   which has motivated our research here.
\begin{varthm*}{Conjecture}\label{conj:main}
   Let $X$ be a smooth projective surface with Picard number $1$ and
   let $L$ be the ample generator of the N\'eron--Severi space with $(L^2)=d$.
   Assume furthermore that $d$ is a non-square. Then
   $$\eps(L;1) \geq \frac{p_0}{q_0}d,$$
   where $(p_0,q_0)$ is the \emph{primitive solution} of the Pell equation \eqnref{eq:Pell}.
\end{varthm*}
   In other words, the Seshadri constants of generators of the ample cone on surfaces with Picard number $1$
   taken at very general points are expected not to lye in the set $\Exc(d;p_0,q_0)$.
\begin{varthm*}{Remark}
   It is well known that there are surfaces such that the equality in
   Conjecture holds, so that the lower bound there cannot be improved
   without additional conditions on $X$, see \cite{Bau99}.
\end{varthm*}
   In Theorems \ref{thm:p0=1} and \ref{thm:p0=2},
   we verify Conjecture \ref{conj:main} for line bundles $L$, whose
   self-intersection $(L^2)$ belongs to one of two infinite sequences of integers.

   We conclude our note with a case by case study of line bundles with low self-intersection
   numbers. This will fill the whole Section \ref{sec:low}.

%*****************************************************************************

\section{General properties of the Seshadri constants}
   In this section we recall properties of the Seshadri constants needed in the sequel.
   For a general introduction to this circle of ideas we refer to the book
   of Lazarsfeld \cite{PAG} and the survey \cite{Primer}.

   We begin with a useful lower bound on the self-intersection of curves moving
   in a non-trivial family, see \cite[Theorem A]{KSS09} and \cite{Bas09}.
   This bound will be applied to the Seshadri curves of a fixed line bundle.
\begin{proposition}[Bounding self-intersection of curves in a family]\label{prop:Xu}
   Let $X$ be a smooth projective surface. Let $(C_t,x_t)$ be a non-trivial
   family of pointed curves $C_t\subset X$ such that for some integer $m\geq 2$
   there is $\mult_{x_t}C_t\geq m$. Then
   $$(C_t^2)\geq m(m-1)+\gon(C_t).$$
\end{proposition}
   The next proposition shows that if $\eps(L;1)$ is relatively small
   compared to $(L^2)$, then there are geometric reasons, see \cite[Theorem]{STG04}.
   We say that $X$ is \emph{fibred by Seshadri curves} if there exists
   a morphism $f:X\to B$ to a curve $B$ such that for a very general point $x\in X$,
   the curve $f^{-1}(f(x))$ computes $\eps(L;x)$.
\begin{proposition}[Fibration by Seshadri curves]\label{prop:fibration by SC}
   Let $X$ be a smooth projective complex surface and let $L$ be an ample
   line bundle on $X$. If
   $$\eps(L;1)<\sqrt{\frac34(L^2)},$$
   then $X$ is fibred by Seshadri curves. In particular $\eps(L;1)$ is an integer.
\end{proposition}
   We record also for further reference the following property of line bundles
   whose Seshadri curves are smooth.
\begin{proposition}[Smooth Seshadri curves]\label{prop:smooth sesh curve}
   Let $X$ be a smooth projective surface and let $L$ be a primitive ample line bundle on $X$
   (i.e. $L$ is not divisible in the Picard group of $X$).
   Assume that $\eps(L;1)$ is computed by smooth curves. Then
   \begin{itemize}
   \item either $X$ is fibred by Seshadri curves,
   \item or $(L^2)=1$.
   \end{itemize}
\end{proposition}
\proof
   Assume to begin with that
   $$\eps(L;1)<\sqrt{(L^2)}.$$
   Let $C_x$ be a smooth curve computing $\eps(L;x)$ in a very general point
   $x\in X$. Then it is
   $$(L\cdot C_x)<\sqrt{(L^2)}.$$
   Combined with the Hodge Index Theorem, this implies
   \begin{equation}\label{eq:HIT for C_x}
      (C_x^2)(L^2)\leq (L\cdot C_x)^2< (L^2)
   \end{equation}
   and hence $(C_x^2)<1$. Since $C_x$ passes through a very general point
   of $X$, it must be $(C_x^2)\geq 0$, which gives finally $(C_x^2)=0$.

   Now, there is a standard argument (see \cite{Nak03}*{Proof of Theorem 2}
   or \cite{STG04}*{Proof of Theorem}) using the countability of components
   of the Hilbert scheme of curves on $X$, which implies that there is at least one dimensional
   algebraic family of curves $\left\{C_x\right\}$. Since $(C_x^2)=0$, two distinct curves $C_x$ and $C_y$
   in this family
   are disjoint. Thus one can define a map from $X$ to the parameter curve $T$, whose
   very general fibers are the curves $C_x$.

   In the remaining case we have
   $$\eps(L;1)=\sqrt{(L^2)}.$$
   Here the assumption that the Seshadri constant \emph{is actually
   computed by a curve} is essential to conclude. Indeed, we have then the equality in \eqref{eq:HIT for C_x}.
   Hence $0\leq (C_x^2)\leq 1$. If $(C_x^2)=0$, we conclude as before. If $(C_x^2)=1$, then
   we have equality in the Hodge inequality, so it must be that $C_x$ and $L$
   are numerically proportional. Since $L$ is primitive, it must be $(L^2)=1$ and we are done.
\endproof
   Proposition \ref{prop:smooth sesh curve} implies immediately the following property
   of line bundles on surfaces with Picard number $1$.
\begin{corollary}\label{cor:Seshadri curves on surf rho 1}
   Let $X$ be a surface with Picard number $1$ such that the Seshadri constant of the ample
   generator $L$ at a very general point $x$ of $X$ is computed by a curve $C_x\in |kL|$
   for some fixed $k$. Then
   \begin{itemize}
      \item either $\mult_xC_x\geq 2$,
      \item or $(L^2)=1$ and $\eps(L;1)=1$.
   \end{itemize}
\end{corollary}
\proof
   This follows from Proposition \ref{prop:smooth sesh curve} since there are
   no fibrations on surfaces with Picard number $1$.
\endproof
   The next example shows in particular that the assumption on the Picard
   number of $X$ in Corollary \ref{cor:Seshadri curves on surf rho 1}
   is essential.
\begin{example}\label{ex:L^2 large eps 1}
   Let $X=\P^1\times\P^1$ and let $L=sH+V$, where $H$ is the class
   of the fiber of the projection from $X$ onto the second factor and $V$ is the
   class of the fiber of the projection onto the first factor. Then
   $(L^2)=2s$ and $\eps(L;x)=1$ for all points $x\in X$. Indeed, the
   fiber in $|H|$ passing through $x$ is the Seshadri curve of $L$ at $x$.
\end{example}

%*****************************************************************************

\section{An application of Okounkov bodies to the Seshadri constants}
   In this section we prove the Main Theorem.
   The proof builds upon ideas of K\"uronya and Lozovanu from \cite{KurLoz14}.
   Relating the Seshadri constants and Okounkov bodies is not new, see e.g. \cites{Ito13, DKMS16}.
   The key insight here is to use the \emph{infinitesimal} approach
   in the form which is a slight  generalization of \cite{KurLoz14}*{Example 4.4}.
   Our approach has been also strongly influenced by works of Nakamaye
   and Cascini \cites{Nak05, CasNak14}. For an introduction to Okounkov bodies
   we refer to the work of Lazarsfeld and Musta{\c{t}}{\u{a}} \cite{LazMus09}.
\proofof{Main Theorem}
   If $\epsilon(L;1)\geq\frac{p}{q}d$, then there is nothing to prove.
   In the remaining case it must be $\epsilon(L;x)<\frac{p}{q}d$ for an \emph{arbitrary} point $x\in X$.
   Let now $x$ be a \emph{general} point of $X$, i.e. a point in which the function
   defined in \eqref{eq:Seshadri function}
   attains its maximal value. Let $C$ be a curve with $a=(L.C)$ and $b=\mult_xC$ computing $\eps(L;x)$,
   i.e. we have
\begin{equation}\label{eq:pq p0q0}
   \eps(L;x)=\frac{a}{b}<\frac{p}{q}d.
\end{equation}
   Note that the integers $a$ and $b$ need not to be coprime.

   If $b=1$, i.e. $\eps(L;x)$ is computed by a \emph{smooth} curve,
   then
   \begin{equation}\label{eq:exc1}
      \eps(L;x)\in\left\{1,2,3,\ldots,\lfloor\sqrt{d}\rfloor\right\}\subset \Exc(d;p,q)
   \end{equation}
   and we are done. (Proposition \ref{prop:smooth sesh curve} provides
   additional information on $X$ and $L$ in this case.)

   So we may assume $b\geq 2$. Then by \cite{KurLoz14}*{Proposition 4.2}, the generic
   infinitesimal Newton--Okounkov polygon $\Delta(L;x)$ is contained in the triangle $\triangle_{OAB}$ with vertices
   at points
$$
 O \equ (0,0)\ ,\ A \equ (a/b,a/b)\ ,\ B \equ (a/(b-1),0)\ .
$$
   Comparing the areas of the two figures, we obtain
\begin{equation}\label{eqn:GINO}
   \frac{a}{b}\cdot \frac{a}{b-1} \equ 2\cdot \Area(\triangle_{OAB}) \dgeq 2\cdot \text{Area}(\Delta(L;x)) \equ d.
\end{equation}
   From \eqref{eq:pq p0q0} and \eqref{eqn:GINO}, we obtain
$$
 \frac{a^2}{b^2} <\frac{p^2 d}{q^2}d \equ \frac{q^2-1}{q^2} d \leq \frac{q^2-1}{q^2} \frac{a^2}{b(b-1)},
$$
which implies
$$
   \frac{b-1}{b}<\frac{q^2-1}{q^2}.
$$
   This inequality can hold if and only if $b<q^2$. This verifies the bound on the multiplicity
   of the Seshadri curve asserted in the Theorem.

   In order to conclude observe that all possible pairs $(a,b)\in\N^2$ with $\tfrac{a}{b}<\tfrac{p}{q} d$ lie in the range
\begin{equation}\label{eq:exc2}
    1\,<\, b \,<\, q^2\ \ \text{and}\ \ b \,\leq\, a \,<\, b\cdot \frac{p}{q} d\ ,
\end{equation}
   hence they are contained in the set $\Exc(d;p,q)$. Note that the inequality \eqnref{eqn:GINO}
   restricts the actual set of possible values even further. We will explore this
   in Section \ref{sec:low}.
\endproof

%*****************************************************************************

\section{Towards the Conjecture}
   In this section we prove that the Conjecture holds for
   two sequences of integers $d$ such that the primitive solution $(p_0,q_0)$
   of the Pell equation \eqnref{eq:Pell} satisfies $p_0\in\left\{1,2\right\}$.
\begin{theorem}[The case $p_0=1$]\label{thm:p0=1}
   Let $d=n^2-1$ for a positive integer $n$. Then the Conjecture
   holds for all polarized pairs
   $(X,L)$, with $X$ a smooth projective surface with Picard number one,
   and $L$ the ample generator of $\Pic(X)$ with $(L^2)=d$, that is
 $$
  \epsilon(L;1) \dgeq \frac{p_0}{q_0}d \equ \frac{n^2-1}{n}\ .
 $$
\end{theorem}
\proof
   For $d=n^2-1$ the primitive solution to Pell's equation is $(p_0,q_0)=(1,n)$.
   Assume to the contrary that for a very general point $x\in X$ there exists
   a curve $C\in |kL|$ for some $k\geq 1$ computing $\eps(L;x)=a/b$
   and $a,b\in\Exc(d;1,n)$. By Corollary \ref{cor:Seshadri curves on surf rho 1} we have $b\geq 2$.
   Thus Proposition \ref{prop:Xu} applies and we have
   \begin{equation}\label{eq:p01}
      k^2(n^2-1)=k^2d=(C^2)\geq b(b-1)+1.
   \end{equation}
   On the other hand
   $$\frac{a}{b}=\frac{L.C}{b}=\frac{kd}{b}=\frac{k(n^2-1)}{b}<\frac{p_0}{q_0}d=\frac{n^2-1}{n}$$
   implies
   \begin{equation}\label{eq:p02}
      b\geq kn+1.
   \end{equation}
   Now it easy to see that \eqref{eq:p01} and \eqref{eq:p02} cannot be
   satisfied simultaneously. Indeed, combining both inequalities we get
   $$k^2(n^2-1)\geq (kn+1)kn+1.$$
\endproof
   In the previous case we had $p_0=1$. Now we pass to the next case, i.e. $p_0=2$.
   Then $d$ is of the form $d=n^2+n$ for some integer $n\geq 1$ and it is $q_0=2n+1$.
\begin{theorem}[The case $p_0=2$]\label{thm:p0=2}
   The Conjecture holds for all integers $d$ of the form $d=n^2+n$ for some $n\geq 1$.
\end{theorem}
\proof
   We assume to the contrary that for a very general (hence any) point on $X$,
   there exists a curve $C_x\in|kL|$ for some $k\geq 1$ such that
   with $a=(L.C_x)=kL^2$ and $b=\mult_{x}C_x$ there is
   $$
      \frac{a}{b}<\frac{2}{2n+1}(n^2+n).
   $$
   Equivalently we have
   \begin{equation}\label{eq:2 lower than pell}
      \frac{k(2n+1)}{2}<b.
   \end{equation}
   The usual argument with the Hilbert scheme of curves revoked in
   the proof of Proposition \ref{prop:smooth sesh curve} implies
   that such curves move in a non-trivial family of dimension at
   least $1$.

   We have $b\geq 2$ by Corollary \ref{cor:Seshadri curves on surf rho 1}.
   Hence, by Proposition \ref{prop:Xu} we have
   \begin{equation}\label{eq:2 xu}
      b(b-1)+1\leq k^2(n^2+n).
   \end{equation}
   Now the argument splits according to the parity of $k$.\\
   \textbf{Case 1.} Assume that $k=2\ell$.
   Then \eqref{eq:2 lower than pell} reads $b>2\ell n+\ell$. This implies
   $$
      b\geq 2\ell n+\ell +1,
   $$
   and in turn we get
   \begin{equation}\label{eq:2 k even better lower q}
      b(b-1)+1\geq (2\ell n+\ell +1)(2\ell n+\ell) =1.
   \end{equation}
   On the other hand from \eqref{eq:2 xu} we get
   \begin{equation}\label{eq:2 k even upper q}
      b(b-1)+1\leq 4\ell^2(n^2+n).
   \end{equation}
   It is elementary to check that \eqref{eq:2 k even better lower q}
   and \eqref{eq:2 k even upper q} together give a contradiction.\\
   \textbf{Case 2.} The case $k=2\ell+1$ follows similarly.
   From \eqref{eq:2 lower than pell} we get
   $b>(2\ell +1)n+\frac{2\ell+1}{2}$ so that
   $$b\geq (2\ell+1)n+\ell+1.$$
   Hence
   $$b(b-1)+1\geq ((2\ell+1)n+\ell+1)((2\ell+1)n+\ell)+1$$
   and this contradicts inequality \eqref{eq:2 xu}
   in this case as well. We leave the details to the reader.
\endproof
   Thus the first remaining case is $p_0=3$. The first $d$
   with $p_0=3$ is $d=7$. We will see in the next section that
   already in this case our approach leaves over some possibilities
   which require additional arguments.
%*****************************************************************************

\section{Line bundles with small self-intersection}\label{sec:low}
   In this section we analyze consequences of the Main Theorem
   on the distribution of values of the Seshadri constants in general points of
   line bundles with fixed low degree both in general and in the $\rho(X)=1$ cases.
\subsection{Line bundles of degree $1$}
   If $L^2=1$, then Theorem of Ein-Lazarsfeld mentioned in the Introduction
   immediately yields $\eps(L;1)=1$. Additionally, in this case it is known that the number of points,
   where $\eps(L;x)$ attains a value strictly less than $1$ is finite.

   The following example shows that there is little hope to obtain any
   classification of line bundles with self-intersection $1$.
\begin{example}\label{ex:line bundles with L^2=1}
   Let $X$ be a smooth projective surface with Picard number $\rho(X)=1$
   and let $L$ be the ample generator on $X$
   with $d=(L^2)$. Let $f:Y\to X$ be the blow up of $X$ at $d-1$ very general points,
   with the exceptional divisor $\E$ (being the union of $N-1$ exceptional curves).
   Then $M:=f^*L-\E$ is an ample line bundle
   with $(M^2)=1$.
\end{example}
\subsection{Line bundles of degree $2$}
   In this case the primitive solution of the Pell's equation is $p_0=2$ and $q_0=3$
   so that
   $$\Exc(2;2,3)=\left\{1,\frac54,\frac65,\frac76,\frac87,\frac97,\frac98,\frac{10}{8}\right\}.$$
   The extremal value $1$ is attained for example by the line bundle $L$ of bidegree $(1,1)$ on $\P^1\times\P^1$.

   If $\rho(X)=1$, then the Conjecture holds by Theorem \ref{thm:p0=2} and we have
   $$\eps(L;1)\geq\frac43.$$
   The value $4/3$ is actually attained on a principally polarized simple abelian surface,
   see \cite{Ste98}*{Proposition 2}.
\subsection{Line bundles of degree $3$}
   The primitive solution of the Pell's equation is now $p_0=1$ and $q_0=2$.
   Hence the exceptional set in this case is
   $$\Exc(3;1,2)=\left\{1,\frac43\right\}.$$

   Let $f:X\to\P^2$ be the blow up of a point $P\in\P^2$ with the exceptional divisor $E$
   and let $H=f^*(\calo_{\P^2}(1))$.
   Then for the line bundle $L=2H-E$ we have $(L^2)=3$ and $\eps(L;1)=1$, the Seshadri
   constant at a point $Q\in X$ being computed by the proper transform of the line
   passing through $P$ and $Q$ (this applies also to $Q$ infinitesimally near to $P$).

   The exceptional value $\frac43$ is excluded by Proposition \ref{prop:fibration by SC}.

   If $\rho(X)=1$, then the Conjecture holds by Theorem \ref{thm:p0=1}, hence $\eps(L;1)\geq\frac32$.

\subsection{Line bundles of degree $5$}
   The primitive solution of Pell's equation is now $p_0=4$ and $q_0=9$.

   The lower bound predicted by the Conjecture equals in this case $\tfrac{20}{9}$,
   while the Main Theorem
   leaves us with the set of $2401$ possible exceptional pairs $(a,b)$ satisfying
$$
   2\dleq b\dleq 80\ \ \text{ and }\ \ b+1\dleq a < \frac{20}{9}b\ .
$$
   By making sure that the pairs above satisfy the inequality (\ref{eqn:GINO}), we reduce the number of exceptions to $41$,
   starting out with
$$
 (4,2), (6,3), (8,4), (10,5), (11,5), (13,6), (15,7), (17,8), \ldots
$$
   and ending with $(151,68)$.
   Thus for a line bundle $L$ with $(L^2)=5$ we have
\begin{equation}\label{eq:list for L2=5}
\text{ either }\   \epsilon(L;1)\geq \frac{20}{9}\ \ \ \text{or}\ \  \epsilon(L;1)\in\st{1, 2,\frac{11}{5},\frac{13}{6},\frac{15}{7},\frac{17}{8},\ldots
}\
\end{equation}
   where the latter set consists of $28$ values (some exceptional pairs give the same value
   of the Seshadri constant).  This list cannot be further reduced with our methods for an arbitrary surface
   $X$ and an arbitrary line bundle $L$ with $(L^2)=5$.
   We show here surfaces with the two least values of $\eps(L;1)$ in the list \eqref{eq:list for L2=5}.
   \begin{example}($N=5$ and $\eps(L;1)=1$)
   Let $f:X\to\P^2$ be the blow up of $\P^2$ in a point $P$. Let as usual $H=f^*(\calo_{\P^2}(1))$
   and let $L=3H-2E$, where $E$ is the exceptional divisor of $f$. Let $x\in X$ be a generic
   point. In particular $x$ is not a point on the exceptional divisor, so that $x$ can
   be viewed also as a point on the projective plane $\P^2$. Let $C_x$ be the line
   joining $x$ and $P$. Then its proper transform $D_x$ on $X$ has the class $H-E$.
   Thus
   $$\epsilon(L;x)\leq \frac{L.D_x}{1}=1.$$
   Since $\epsilon(L;1)\geq 1$ by the Theorem of Ein-Lazarsfeld, we conclude that $\epsilon(L;1)=1$
   in this case.
\end{example}
\begin{example}($N=5$ and $\eps(L;1)=2$)
   Let $f:X\to\P^1\times\P^1$ be the blow up of $\P^1\times\P^1$ at $3$ general
   points $P,Q,R$ with exceptional divisors $E,F,G$. Let $H=f^*(\calo_{\P^1\times\P^1}(1,1)$
   and let $L=2H-E-F-G$. Let $x$ be a generic point on $X$. In particular $x$ is not contained
   in the union of the exceptional divisors $E\cup F\cup G$. Let $C_x$ be a curve of type $(1,1)$
   in $\P^1\times\P^1$ passing through $P,Q$ and $x$. Then its proper transform $D_x$
   on $X$ has the class $H-E-F$. Hence
   $$\epsilon(L;x)\leq\frac{L.D_x}{1}=2.$$
   With a little more care one can show that in this case indeed $\epsilon(L;1)=2$.
\end{example}

   It turns out that we can reduce further the number of possibilities
   by imposing the condition $\rho(X)=1$ on $X$. Let $L$ be an ample generator
   with $(L^2)=5$. Then for any curve $C\in |kL|$ for some positive integer $k$
   $5$ divides $a=(L\cdot C)$.
   This leaves only the pairs
$$
 (10,5), (15,7), (35,16), (55,25), (75,34)\ .
$$

   In summary, if $X$ is a smooth projective surface with Picard number one, $L$ an ample line bundle on $X$ with $(L^2)=5$, then
$$
 \text{ either }\   \epsilon(L;1)\geq \frac{20}{9}\ \ \ \text{or}\ \
 \epsilon(L;1)\,\in\, \st{2, \frac{11}{5}, \frac{15}{7}, \frac{35}{16}, \frac{75}{34}}\ .
$$

\subsection{Line bundles of degree $6$}
   The primitive solution of Pell's equation is now $p_0=2$ and $q_0=5$.
   A computer count shows that the set $\Exc(6;2,5)$ has 252 elements
   $$\Exc(6)=\left\{1,\frac{25}{24},\frac{24}{23},\frac{23}{22},\ldots,\frac{31}{13},\frac{43}{18},\frac{55}{23}\right\}.$$
   A considerable number of these values can be discarded using Proposition \ref{prop:fibration by SC}.
   The modified set $\Exc'(6;2,5)$ still contains $51$ elements
   $$\Exc'(6)=\left\{\frac{17}{8},\frac{49}{23},\frac{32}{15},\ldots,\frac{31}{13},\frac{43}{18},\frac{55}{23}\right\}.$$
   However, if $\rho(X)=1$, then Theorem \ref{thm:p0=2} implies that the Conjecture holds,
   so that all these $51$ possible values are also discarded under this assumption. This underlines
   the power of Theorem \ref{thm:p0=2}. It is also worth to remark here that simple application
   of Proposition \ref{prop:Xu} (i.e. without taking into account that all involved numbers are
   integers and subject to certain divisibilities) would leave the following set of $6$ exceptional
   values
   $$\Exc''(6;2,5)=\left\{\frac94, \frac73, \frac{26}{11}, \frac{19}8, \frac{31}{13}, \frac{43}{18} \right\}.$$
   This shows that arithmetic flavor arguments as in the proof of Theorem \ref{thm:p0=2},
   even though simple, are in fact inevitable.

\subsection{Line bundles of degree $7$}
   The primitive solution of Pell's equation is now $p_0=3$ and $q_0=8$.
   We consider now only surfaces with $\rho(X)=1$.
\begin{theorem}
   Assume that $X$ is a surface with $\rho(X)=1$ and let $L$ be an ample
   generator with $(L^2)=7$. Then
   $$\mbox{either } \eps(L;1)\geq \frac{21}{8}\;\mbox{ or }\; \eps(L;1)=\frac{28}{11}.$$
\end{theorem}
\begin{proof}
   Taking into account that $\rho(X)=1$ and the divisibility condition $7|a$ the list of possible
   exceptional values of $\eps(L;1)$ is reduced to
$$
 (7,3), (28,11), (49,19).
$$
   We show how to exclude the pairs $(7,3)$ and $(49,19)$. Since the Seshadri curve $C$
   is singular in both cases, these curves form a $2$-dimensional family. Then by Proposition \ref{prop:Xu} with
   either $(C^2)=7$ and $q=3$, or $(C^2)=343$ and $q=19$  we obtain $\gon(C)=1$. Hence $X$ is covered by
   rational curves. As these curves intersect, $X$ is actually a rational
   surface. But the assumption $\rho(X)=1$ forces $X$ to be $\P^2$.
   This contradicts the assumption that the ample generator of the Picard
   group has degree $7$.
\end{proof}

\subsection{Line bundles of degree $8$}
   The primitive solution of Pell's equation is now $p_0=1$ and $q_0=3$.
   The set $\Exc(8;1,3)$ contains $37$ elements ranging
   from $1$ to $21/8$. Applying Proposition \ref{prop:fibration by SC} the list
   reduces to
   $$\Exc'(8;1,3)=\left\{ \frac52, \frac{18}{7}, \frac{13}5, \frac{21}{8} \right\}.$$
   Assuming additionally that $\rho(X)=1$, the list gets empty in accordance to Theorem \ref{thm:p0=1}.

%*****************************************************************************
\paragraph*{Acknowledgement.} This research was started during the conference
   ``Recent advances in Linear series and Newton-Okounkov bodies'' held in Padova
   in 2015 and finished during the workshop
   ``Positivity and Valuations'' held in Barcelona in 2016. We would like
   to thank the organizers of both meetings for providing platforms for
   stimulating discussions and great working conditions. We are indebted to Alex K\"uronya
   and Victor Lozovanu for explaining their recent works to us.\\
   Our research was partially supported by National Science Centre, Poland, grant
   2014/15/B/ST1/02197.
%*****************************************************************************

%\begin{thebibliography}{99}\footnotesize

%\end{thebibliography}

%***************************************************************************** % Addresses

\bigskip \small

\bigskip
   {\L}ucja Farnik,
   Jagiellonian University, Institute of Mathematics, {\L}ojasiewicza 6, PL-30348 Krak\'ow, Poland

\nopagebreak
   \textit{E-mail address:} \texttt{Lucja.Farnik@im.uj.edu.pl}

\bigskip
   Tomasz Szemberg,
   Department of Mathematics, Pedagogical University of Cracow,
   Podchor\c a\.zych 2,
   PL-30-084 Krak\'ow, Poland.

   Current Address:
   Polish Academy of Sciences,
   Institute of Mathematics,
   \'Sniadeckich 8,
   PL-00-656 Warszawa, Poland

\nopagebreak
   \textit{E-mail address:} \texttt{tomasz.szemberg@gmail.com}

\bigskip
   Justyna Szpond,
   Department of Mathematics, Pedagogical University of Cracow,
   Podchor\c a\.zych 2,
   PL-30-084 Krak\'ow, Poland

\nopagebreak
   \textit{E-mail address:} \texttt{szpond@up.krakow.pl}

\bigskip
   Halszka Tutaj-Gasi\'nska,
   Jagiellonian University, Institute of Mathematics, {\L}ojasiewicza 6, PL-30348 Krak\'ow, Poland

\nopagebreak
   \textit{E-mail address:} \texttt{Halszka.Tutaj@im.uj.edu.pl}

%*****************************************************************************

\end{document}